\documentclass[11pt]{article}
\usepackage{amsfonts}
\usepackage{amsthm}
\usepackage{amsmath}
\usepackage{graphicx}
\usepackage[usenames,dvipsnames]{xcolor}
\usepackage{enumerate}
\usepackage{hyperref}

\newtheorem{theorem}{Theorem}[section]
\newtheorem{corollary}[theorem]{Corollary}
\newtheorem{proposition}[theorem]{Proposition}
\newtheorem{lemma}[theorem]{Lemma}
\newtheorem{remark}[theorem]{Remark}

\newcommand{\be}{\begin{equation}}
\newcommand{\ee}{\end{equation}}
\newcommand{\ben}{\begin{equation*}}
\newcommand{\een}{\end{equation*}}

\newcommand{\bfeins}{{\bf 1}}
\newcommand{\veps}{\varepsilon}
\newcommand{\topr}{\stackrel{\mathrm{P}}{\longrightarrow}}
\newcommand{\todr}{\stackrel{\mathrm{D}}{\longrightarrow}}

\newcommand{\eqdr}{\stackrel{\mathrm{D}}{=}}
\newcommand{\R}{\Bbb{R}}
\newcommand{\N}{\Bbb{N}}

\newcommand{\rmd}{{\rm d}}
\newcommand{\rmi}{{\rm i}}

\newcommand{\RR}{\mathbb{R}}

\newcommand{\wt}{\widetilde}
\newcommand{\wh}{\widehat}

\newcommand{\dto}{\downarrow}

\newcommand{\pibar}{\overline{\Pi}}
\newcommand{\pibarinv}{\overline{\Pi}^{\leftarrow}}
\newcommand{\pibarpinv}{\overline{\Pi}^{+,\leftarrow}}
\newcommand{\pibarminv}{\overline{\Pi}^{-,\leftarrow}}
\newcommand{\pibarpminv}{\overline{\Pi}^{\pm,\leftarrow}}
\newcommand{\thetabar}{\overline{\Theta}}

\newcommand{\lambar}{\overline{\Lambda}}
\newcommand{\laminv}{\overline{\Lambda}^{\leftarrow}}
\newcommand{\lampinv}{\overline{\Lambda}^{+,\leftarrow}}
\newcommand{\lamminv}{\overline{\Lambda}^{-,\leftarrow}}

\newcommand\MPCPS{\textit{Math. Proc. Camb. Phil. Soc.}}
\newcommand\AOS{\textit{Ann. Statist.}}

\begin{document}

\title{Tightness and Convergence of Trimmed L\'evy Processes to Normality at Small Times}

\author{Yuguang Fan \footnote{Email : \href{mailto:yuguang.fan@unimelb.edu.au}{yuguang.fan@unimelb.edu.au}} \\
{\normalsize ARC Center of Excellence for Mathematical and Statistical Frontiers.}\\
{\normalsize School of Mathematics and Statistics,
The University of Melbourne, Australia.}
 }

\maketitle

\begin{abstract}
For non-negative integers $r$, $s$, let $^{(r,s)}X_t$ be the L\'evy process $X_t$ with the $r$ largest positive jumps and the $s$ smallest negative jumps up till time $t$ deleted, and let $^{(r)}\wt X_t$ be $X_t$ with the $r$ largest jumps in modulus up till time $t$ deleted. Let $a_t \in \R$ and $b_t>0$ be non-stochastic functions in $t$.  We show that the tightness of $({}^{(r,s)}X_t - a_t)/b_t$  or $({}^{(r)}\wt X_t - a_t)/b_t$ as $t\dto 0$  implies the tightness of all normed ordered jumps,  hence the tightness of the untrimmed process $(X_t -a_t)/b_t$ at $0$. We use this to deduce that the trimmed process $({}^{(r,s)}X_t - a_t)/b_t$  or $({}^{(r)}\wt X_t - a_t)/b_t$
converges to $N(0,1)$ or  to a degenerate distribution as $t\dto 0$ if and only if $(X_t-a_t)/b_t $ converges to $N(0,1)$ or to the same degenerate distribution, as $t \dto 0$.
\end{abstract}


\section{Introduction}
L\'evy processes can be seen as continuous time analogues of random walks. Historically motivated by robust statistics, the concept of trimming has been thoroughly explored in the random walks setting to assess the effect of outliers. Here we construct an analogous process by removing a finite number of largest jumps from a L\'evy process. For large time behaviour, i.e. as $t\to \infty$, the trimmed L\'evy process exhibits a similar structure to the trimmed sums of independent and identically distributed random variables. 
In this paper, however, we are concerned with small time convergence properties. As $t \to \infty$, an increasing number of jumps with bigger magnitude come into consideration for being removed, but as $t \dto 0$, jumps of bigger size are gradually excluded from being deleted.
%%removed in the trimming procedure. 
This makes trimming at small times a non-trivial effort with no exact random walk analogy. 

The idea of removing jumps from a L\'evy process is not at all new. Rosi\'nski \cite{ros2001} made use of ``thinning'' to generate one L\'evy process from another by removing jumps stochastically. In comparison, the  processes introduced by  Buchmann, Fan and Maller \cite{bfm14} have a more deterministic flavour; jumps  are removed according to their sizes. 
The resulting trimmed processes no longer have independent stationary increments, hence are not L\'evy processes. But their distributions can be written as mixtures of a truncated infinitely divisible distribution with a gamma random variable. This was done in \cite{bfm14}, where representation formulae for the distribution of trimmed processes joint with their order statistics and quadratic variation were derived. %% for positive and modulus trimming.
These representations permit techniques for L\'evy processes to be carried over to the trimmed processes. In Section \ref{sect:trun}, as preparatory material for the proofs of the main results, we revisit and extend the results in \cite{bfm14} to asymmetrical trimming. 

%theoretical motivation
The focus of the present paper is the small time domain of attraction problem for L\'evy processes, which has recently received much attention.
Maller and Mason \cite{MM2010} gave necessary and sufficient conditions on the canonical measure of a L\'evy process for it to converge,  after appropriate centering and norming, to a stable law as $t\dto 0$. 
(See also Maller and Mason \cite{MM2010} and \cite{MM08}, and de Weert \cite{deW2003}.)
The question then arises as to the effect of removing large jumps of the process on this kind of convergence. Fan \cite{fan2014st} investigated this for the case of attraction to a non-normal stable law as $t\dto 0$,
and gave a complete characterisation.
%% of this situation. 
There it was shown  that ``lightly trimmed" L\'evy processes, i.e., after trimming off a finite number of large jumps, converge at small times with appropriate centering and norming to a non-degenerate non-normal law if and only if the original L\'evy process converges to an almost surely finite,
 non-degenerate, non-normal, limit random variable. 

The purpose of the present paper is to extend this result to the case of a normal limit, where again a complete characterisation can be given: 
light trimming produces a normal limiting distribution if and only if the process is in the domain of attraction of the normal, as $t\dto 0$.
 Taken together with the results of \cite{fan2014st}, this provides a complete solution to the  domain of attraction problem for trimmed L\'evy processes.

Our findings can be seen as a continuation of the classical precedent in random walks, and we borrow from a rich repertoire of ideas in the random walk literature.
It has been shown there that the convergence of normed, centered random walk to a finite, non-degenerate random variable implies the convergence of the lightly trimmed sum (see for example Darling \cite{darling52}, Hall \cite{hall78}  and Mori \cite{mori84}). However, the converse is known to be a much harder problem. Maller \cite{mal82} first gave a partial converse by showing that when the trimmed sum 
converges to normality, under the assumption of a continuous and symmetrically  distributed increment, the untrimmed sum also converges to normality. Mori \cite{mori84} completed the proof for the general case without extra assumptions for asymptotic normality, only, and admitted the difficulties in proving a similar result for a non-normal limit.  

In 1993, Kesten \cite{kesten93} proved the most general case by showing that the convergence in distribution of normed and centered lightly trimmed and untrimmed random walks $S_n$ are equivalent as $n \to \infty$. In Fan \cite{fan2014st} we extended the Kesten analysis to the small time L\'evy domain, thus characterising the trimmed domain of attraction for non-normal laws.
 Although motivated by Kesten's method, some quite different techniques had to be developed to deal with the small time convergences.
    These results apply to a wide class of processes of practical interest which have non-trivial domains of attraction, for example, the tempered stable processes, Lamperti stable processes and the normal variance-mean mixture processes. We refer to \cite{BN1977}, \cite{BN1998}, \cite{caballero2008} and Fan \cite{fan_thesis} for further details. 
   
The present paper completes the picture with the non-normal limits.
 Our main results are in Theorems \ref{tighteqv} and \ref{convNc} below. They analyse the effect of light trimming %%in a weak convergence sense, 
 on the tightness as well as the asymptotic normality at $0$ of a normed and centered L\'evy process.

 \section{Main Results}
 Our setup is as follows. Let $(X_t)_{t\ge 0}$ be a real valued L\'evy process with canonical triplet $(\gamma,\sigma^2,\Pi)$,
thus having characteristic function $Ee^{\rmi \theta X_t}= e^{t\Psi(\theta)}$, $t\ge 0$,
$\theta \in \R$, with characteristic exponent
\ben\label{ce}
\Psi(\theta):= \rmi \theta \gamma - \frac 1 2 \sigma^2 \theta^2 +\int_{\R}
\left(e^{\rmi \theta x}-1-\rmi \theta x{\bf 1}_{\{|x|\le 1\}}\right)
\Pi(\rmd x),
\een
where $\gamma\in\R$, $\sigma^2\ge 0$. %, $\Pi$ a L\'evy measure on $\R$.
 Here $\Pi$ is a Borel measure on $\RR_* : = \RR\setminus \{0\}$ with $\int_{\R^*} 1 \wedge x^2 \Pi(\rmd x) < \infty$  and $\Pi((-x,x)^c)<\infty$ for all $x>0$.

Denote the jump process of $X$ by $(\Delta X_t)_{t\ge 0}$, where $\Delta X_t = X_t-X_{t-}$, $t>0$, with  $\Delta X_0 \equiv 0$. In particular, denote the positive jumps by $\Delta X_t^+ = \Delta X_t \vee 0$ and the magnitudes of the negative jumps by $\Delta X_t^- = (-\Delta X_t)\vee 0$.
Note that $(\Delta X_t^+)_{t\ge 0}$ and $(\Delta X_t^-)_{t \ge 0}$ are non-negative independent processes. 
For any integers $r,s  > 0$, let $\Delta X_t^{(r)}$ be the $r^{th}$ largest positive jump, and let $\Delta X_t^{(s),-}$ be the $s^{th}$ largest jump in $\{\Delta X_s^{-}, 0<s\le t\}$, which corresponds to the magnitude of the $s^{th}$ smallest negative jump. We further write $\wt{\Delta X}_t^{(r)}$ to denote the $r^{th}$ largest jump in modulus up to time $t$.  In what follows, $\Delta X_t^{(r),+}$ and $\Delta X^{(r)}_t$ are written interchangeably.  For a formal definition of the ordered jumps, allowing tied values, we refer to Buchmann et al. \cite{bfm14} Section 2.1.
The trimmed versions of  $X$ are defined as
\begin{equation}\label{trims}
{}^{(r, s)} X_t:= X_t- \sum_{i=1}^r {\Delta X}_t^{(i)} + \sum_{j=1}^s \Delta X_t^{(j),-}
 \quad {\rm and} \quad
{}^{(r)}\wt X_t:= X_t- \sum_{i=1}^r \wt{\Delta X}_t^{(i)},
\end{equation} which are termed asymmetrical trimming and modulus trimming respectively. 
By letting $r = 0$ or $s = 0$ in asymmetrical trimming, we obtain one-sided trimmed processes, 
\begin{equation}\label{1s}
 {}^{(r)} X_t:= X_t- \sum_{i=1}^r {\Delta X}_t^{(i)} ,
 \quad {\rm and} \quad
{}^{(s,-)} X_t:= X_t + \sum_{i=1}^s {\Delta X}_t^{(i),-}.
\end{equation}  

Set 
\[{}^{(0,0)}X_t = {}^{(0)} \wt X_t ={}^{(0)}  X_t = {}^{(0,-)}  X_t = X_t. 
\]
The positive, negative and two-sided tails of L\'evy measure $\Pi$ are
\ben\label{pidef}
\pibar^+(x):= \Pi\{(x,\infty)\},\ \pibar^-(x):= \Pi\{(-\infty,-x)\},\
{\rm and} \ \pibar(x):=\pibar^+(x)+\pibar^-(x), \ x>0.
\een
The restriction of $\Pi$ on $(0,\infty)$ is $\Pi^+$. Let $\Pi^-(\cdot) = \Pi(-\cdot)$ and $\Pi^{|\cdot|}  = \Pi^+ + \Pi^-$.
For each $x > 0$, denote the truncated mean and second moment functions by 
\begin{equation*}\label{vdef}
 \nu(x) = \gamma - \int_{x< |y|\le 1} y \Pi(\rmd y), \quad \text{and } \quad V(x) = \sigma^2 + \int_{|y|\le x} y^2 \Pi(\rmd y).
\end{equation*}

Throughout the paper, we assume $\pibar(0+) = \infty $ when dealing with modulus trimming and $\pibar^+(0+) = \infty$ or $\pibar^-(0+) = \infty$ (or both when appropriate) when dealing with one-sided trimming. In particular, these mean $V(x) > 0$ for all $x > 0$, and they ensure there are infinitely many jumps $\Delta X_t$, $\Delta X_t^\pm$, a.s., in any bounded interval of time. 

Analytical conditions for a L\'evy process to be in the domain of attraction of a normal law as $t \dto 0$ or $t \to \infty$ were studied in Doney and Maller \cite{dm02}. $X_t$ is said to be in the \emph{domain of attraction of the normal law} at $0$ if there exist some centering and norming functions $a_t \in \R$ and $b_t > 0$ such that
\be\label{DN}
\frac{X_t-a_t}{b_t} \to N(0,1)  \quad \text{as} \quad  t \dto 0  ,  \quad \text{equivalently, if }  \quad 
\frac{x^2\pibar(x)}{V(x)} \to 0 \quad \text{as} \quad x \dto 0. 
\ee
Here $N(0,1)$ denotes a standard normal random variable.

When \eqref{DN} holds, the norming function $b_t$ is regularly varying with index $1/2$ at $0$ and the truncated second moment function $V(x)$ is slowly varying at $0$. For the definition and properties of regular variation, we refer to \cite{BGT87}. 
At small times, the centering function $a_t$ can be chosen to be $0$, i.e. $X_t$ in the domain of attraction of the normal law ($X_t \in D(N)$) at $0$ is equivalent to $X_t$ in the\emph{ centered domain of attraction of the normal law} $(X_t \in D_0(N))$ at $0$ (see Maller and Mason \cite{MM2010}). 
%Maller and Mason \cite{MM2010} further give a subsequential version of this result.

For given non-stochastic functions $a_t \in \R$ and $b_t > 0$, abbreviate the various centered and normed processes as 
\begin{equation*}\label{cntrim}
S_t := \frac{X_t-a_t}{b_t} , \quad 
  {}^{(r,s)}  S_t := \frac{{}^{(r,s)}  X_t-a_t}{b_t}\quad \text{and}\quad
 {}^{(r)} \wt S_t := \frac{{}^{(r)} \wt X_t-a_t}{b_t} .
\end{equation*}
Also denote the one-sided versions (refer to \eqref{1s}) as
\[ {}^{(r)}S_t : = \frac{{}^{(r)} X_t-a_t}{b_t} \quad \text{and} \quad  {}^{(s,-)}S_t : = \frac{{}^{(s,-)}  X_t-a_t}{b_t} .
\]

We pursue a compactness argument by first proving that $(S_t)$ is relatively compact as $t \dto 0$ if $^{(r,s)}S_t$ or $^{(r)}\wt S_t$ is. This will imply that each subsequence of $(S_t)$ has a further subsequence convergent in distribution. Then we establish that each convergent subsequence has to converge to the same limit when $^{(r,s)}S_t$ or $^{(r)}\wt S_t$ has a normal or degenerate limit as $t \dto 0$. 

%Note that it is not possible to achieve \eqref{DN} directly from the assumption that the trimmed processes converges to normal or degenerate distribution. This is due to the fact that 

The idea of the proof is inspired by Mori \cite{mori84} in the random walks literature, but we apply it to the continuous setting in the small time framework where some notable differences occur. 
Before proving the asymptotic normality result, we establish equivalent conditions for a sequence of normed and centered L\'evy processes to be relatively compact. Since we are dealing with $X_t$ on the real line, we can instead prove that, if $^{(r,s)}S_t$ or $^{(r)}\wt S_t$ is tight at $0$, then $S_t$ is tight at $0$, i.e.
\[ \lim_{x \to \infty} \limsup_{t\dto 0}P(|S_t| > x ) = 0. 
\]

Henceforth we state theorems for both asymmetrical and modulus trimmed processes but only give detailed proofs for one type of trimming. All statements are also true for one-sided trimmed processes, as special cases of the asymmetrical trimmed process by taking either $r=0$ or $s =0$. 
%Mori \cite{mori84} dealt only with modulus trimming while Kesten \cite{kesten93} dealt only with modulus and one-sided trimming. 
Theorem \ref{tighteqv} gives a thorough characterisation of the tightness of trimmed processes, the ordered jumps and the untrimmed process.

%----Statement of Main Theorem 1 about tightness-------

\begin{theorem}\label{tighteqv}
{\rm (a)} Fix $r = 0,1,2, \ldots$ and $s = 0 ,1,2, \ldots$. Suppose that $({}^{(r,s)}S_t)$ is tight as $t\dto 0$ for some $a_t \in \R$ and $b_t > 0$. Then the following hold.
\begin{enumerate}[\rm (i) ]
%\item\label{rsst}
% $(^{(r,s)}S_t)$ is tight at $0$.

\item \label{ordersrp}
 $(\Delta X_t^{(j)}/b_t)$ is tight at $0$ for all $j \in \N$, equivalently, 
 \[\lim_{x\to \infty} \limsup_{t\dto 0} t\pibar^+(x b_t)  = 0.\]

\item \label{ordersrm}
 $(\Delta X_t^{(k), -}/b_t)$ is tight at $0$ for all $k \in \N$, equivalently, 
 \[
 \lim_{x\to \infty} \limsup_{t\dto 0} t\pibar^-(x b_t)  = 0.
 \]

\item\label{rst}
 $(^{(j)}S_t)$ is tight at $0$ for all $j = 1, 2, \ldots$.
\item \label{st-}
 $({}^{(k,-)}S_t)$ is tight at $0$ for all $k = 1, 2, \ldots$.

\item \label{st}
 $(S_t)$ is tight at $0$.
 
\end{enumerate}

{\rm (b)} Suppose $({}^{(r)}\wt S_t)$ is tight at $0$ for some $a_t \in \R$ and $b_t > 0$. Then $(S_t)$ is tight at $0$,  
 $(\wt{\Delta X}_t^{(j)}/b_t)$ is tight at $0$ for some (hence all) $j \in \N$ and $\lim_{x\to \infty} \limsup_{t\dto 0} t\pibar(x b_t)  = 0$. 

%The same statements hold true with $^{(r)}S_t$, $\Delta X_t^{(r)}$ and $\pibar^+$ replaced by $^{(r)}\wt S_t$, $\wt{\Delta X}_t^{(r)}$ and $\pibar$ respectively.\\
%
\end{theorem}

With the help of Theorem \ref{tighteqv} we can prove our main result in Theorem \ref{convNc}.

%----Statement of Main theorem 2on the converse in the normal case---
\begin{theorem}\label{convNc}
Suppose $\pibar(0+) = \infty$. There exist non-stochastic functions $a_t$ and $b_t > 0$ such that, as $t \dto 0$, for any $r,s \in \N$,
\begin{equation}\label{convNm1}
\frac{X_t-a_t}{b_t} \todr N(0, 1) \ \text{ {or a degenerate distribution}},
\end{equation}
if and only if

\begin{equation}\label{convNm+}
\frac{^{(r,s)} X_t-a_t}{b_t} \todr N(0,1) \ \text{{or a degenerate distribution}},
\end{equation}
or equivalently,
\begin{equation*}\label{convNm2}
\quad \frac{^{(r)}\wt X_t-a_t}{b_t} \todr N(0,1)\ \text{{or a degenerate distribution}}.
\end{equation*}
%The same $b_t$ can be used in each case.
%When \eqref{convNm1} or \eqref{convNm2} or \eqref{convNm+} holds, we further have functional convergence of the interpolated processes to standard Brownian motion.
\end{theorem}

%Note that when $s = 0$ we have the usual one sided trimming of $r$ largest jumps, 
%\begin{equation}\label{conv-1s}
%\frac{^{(r)} X_t-a_t}{b_t} \todr N(0,1)\quad {\rm or} \quad 1 .
%\end{equation}

\noindent\textbf{Outline of the Proof}

To show tightness in Theorem \ref{tighteqv}, we make use of a key inequality (Prop. \ref{keyIn}) in Section \ref{sect:ineq} that gives an upper bound to the distribution of the trimmed process. 
Before that, in Section \ref{sect:trun}, we investigate the limit of a truncated L\'evy process as $t \dto 0$, allowing a Poisson number of possible tied values in the jumps.  In Section \ref{sect:prf},  by an important estimate on the tail probability of a L\'evy process in Sato (\cite{sato99}), we prove Theorem \ref{convNc} by showing that each convergent subsequence has the same normal or degenerate limit at $0$. Some auxiliary  results concerning the quadratic variation process of $X$ and the domain of partial attraction of the normal distribution are in Section \ref{sect:aux}.

\section{The Truncated Process}\label{sect:trun}
Here we outline the notation and representation needed from \cite{bfm14} and extend them to the asymmetrically trimmed case.
By the L\'evy-It\^o decomposition (Theorem 19.2, p.120 in \cite{sato99}), we can write 
\ben\label{levyito}
X_t  = \gamma t + \sigma B_t + X_t^{J},
\een
where  $(B_t)$ is a standard Brownian motion and the compensated jump process is
\[
X_t^{ J}  = a.s. \lim_{\veps \dto 0} \left(\sum_{0<s\le t} \Delta X_s {\bf 1}_{\{|\Delta X_s | > \veps \}} - t\int_{\veps < |x| \le 1} x \, \Pi(\rmd x) \right).
\]
%We will need notation and properties of inverse functions of $\pibar$ and $\pibar^\pm$. 
Define the right-continuous inverse of a nonincreasing monotone function $f:(0,\infty)\mapsto [0,\infty)$ as
\ben\label{linv}
f^\leftarrow(x)=\inf\{ y>0:   f(y) \le x   \},\ x>0.
\een

%The following setup and results are taken from Buchmann et al. \cite{bfm14}.
We introduce three families of processes, indexed by $v>0$, truncating jumps from sample paths of $X_t^J$.
Let $v,t>0$. When $\pibar(0+)=\infty$, set
\begin{equation*}\label{2.1b}
 X_t^{+,v}:=X_t^{J}-\sum_{0<s\le t} \Delta X_s\;\bfeins_{\{\Delta X_s\ge \pibarpinv(v)\}}, \quad
 X_t^{-,v}:=X_t^{J}-\sum_{0<s\le t} \Delta X_s \bfeins_{\{\Delta X_s\le -\pibarminv(v)\}}, 
\end{equation*}
and for the modulus case,
%  we truncate from the original process $X_t$ instead of its jump process, i.e.
\begin{equation}\label{2.2b}
\wt X_t^v:=X_t-\sum_{0<s\le t} \Delta X_s\;\bfeins_{\{|\Delta X_s|\ge \pibarinv(v)\}}.
\end{equation}
Under the assumption $\pibar(0+)=\infty$,  $(X_t^{\pm,v})_{t\ge 0}$ and $(\wt X_t^v)_{t\ge 0}$ are well defined L\'evy processes with canonical triplets, respectively,
\begin{equation*}\label{tripasy}
\left(   \mp {\bf 1}_{\{\pibar^{\pm, \leftarrow}(v)\le 1\}}\int_{\pibar^{\pm, \leftarrow}(v)\le x \le 1}x \Pi^\pm(\rmd x),\  0, \,  \Pi^\pm(\rmd x){\bf 1}_ {0<x<\pibar^{\pm, \leftarrow}(v)} \right),
\end{equation*}
and
\begin{equation}\label{trip2}
\left(\gamma-{\bf 1}_{\{\pibarinv(v)\le 1\}}\int_{\pibarinv(v)\le |x| \le 1}x \Pi(\rmd x),\,  \sigma^2,\,  \Pi(\rmd x){\bf 1}_{\{|x|<\pibarinv(v)\}}\right).
\end{equation}

Theorem 2.1 of \cite{bfm14} uses a pathwise construction method to derive representations for the distributions of $ {}^{(r)}X_t$ and ${}^{(r)}\wt X_t$ jointly with their corresponding largest jumps, $\Delta X_t^{(r)}$ and $\wt {\Delta X}_t^{(r)}$. We extend these expressions to the asymmetrically trimmed process ${}^{(r,s)}X_t$ joint with both positive and negative ordered jumps $\Delta X_t^{(r)}$ and $\Delta X_t^{(s),-}$. Note $X_t^J = X_t^+ - X_t^-$ where $X_t^\pm$ are the compensated sums of positive and negative jumps respectively. We can trim these to get  $^{(r,s)}X_t   =\gamma t + \sigma Z_t+ {}^{(r)}X_t^+ - {}^{(s)}X_t^-$, where ${}^{(r)}X_t^+$ and $ {}^{(s)}X_t^-$ are defined analogously as in \eqref{1s}. These processes are non-negative and independent of each other. Therefore the positive and negative jump processes can be treated independently.

For each $r, s \in \N$, let $\Gamma_r$ and $\wt \Gamma_s$ be independent standard Gamma random variables with parameters $r$ and $s$, independent of $(X_t)_{t\ge 0}$. Let $(Y_t)_{t\ge 0}$, $(Y_t^\pm)_{t\ge 0}$ be independent Poisson processes with unit mean, independent from $X$, $\Gamma$, $\wt \Gamma$. On the assumption that $\pibar^+(0+)= \pibar^-(0+) = \infty$, by Theorem 2.1 in \cite{bfm14}, for each $t > 0$,
\begin{align}\label{rand_dis3}
&\left({}^{(r,s)} X_t, \, {\Delta X}_t^{(r)}, \, \Delta X_t^{(s),-}\right) \nonumber \\
&\eqdr \left(  X_t^{u,v}+ G_t^{+,v}- G_t^{-,u},\,  \pibarpinv\left(v\right), \pibarminv\left(u\right)  \right) \bigg|_{v=\Gamma_r/t, u = \wt\Gamma_s/t},
\end{align}
where for $w > 0$,
\[
G^{\pm, w}_t = \pibar^{\pm,\leftarrow}(w) Y_{t\rho_\pm(w)}  \quad \text{and} \quad 
\rho_\pm(w) = \pibar^\pm(\pibar^{\pm, \leftarrow}(w)-)-w 
\]
and for each $u>0$, $v> 0$,
\ben\label{2.7a}
X^{u,v}_t := \gamma t + \sigma Z_t +  { X}_t^{ +, v}-{ X}_t^{-, u}
\een is infinitely divisible with characteristic triplet
\[ \left(\gamma_{u,v},\,  \sigma^2,\,  \Pi(\rmd x){\bf 1}_{\{-\pibarminv(u)< x <\pibarpinv(v)\}}\right).
\] Here
\ben\label{gmauv}
\gamma_{u,v} =\gamma-{\bf 1}_{\{\pibarpinv(v)\le 1\}}\int_{\pibarpinv(v)\le x \le 1}x \Pi(\rmd x) + {\bf 1}_{\{\pibarminv(u)\le 1\}}\int_{\pibarminv(u)\le x \le 1}x \Pi^-(\rmd x).
\een
The processes $G_t^{\pm, w}$ and $\wt G_t^v$ (in \eqref{2rrep_1}) are Poisson processes resulting from possible tied values in the ordered jumps. For completeness, we quote next the representation of the modulus trimmed process from \cite{bfm14} before proceeding to the proofs.

For each $v>0$, recall the modulus truncated process
$(\wt X^{v}_t)_{t\ge 0}$  in  \eqref{2.2b} with canonical triplet 
\begin{equation}\label{trip1}
\left(\wt \gamma_v,\  \sigma^2,\  \Pi(\rmd x){\bf 1}_{\{|x|<\pibarinv(v)\}}\right),
\end{equation} where $\wt \gamma_v = \gamma-{\bf 1}_{\{\pibarinv(v)\le 1\}}\int_{\pibarinv(v)\le |x| \le 1}x \Pi(\rmd x)$ as defined in \eqref{trip2}.
Then, for each $t>0$ and $r \in \N$,
\begin{equation}\label{2rrep_1}
\left(^{(r)}\wt X_t,\,|\wt{\Delta X}_t^{(r)}|\right)
\eqdr
\left(\wt X_t^{v}+\wt G_t^v, \, \pibarinv\left(v\right)\right) \bigg|_{v=\Gamma_r/t},
\end{equation} where $\wt G_t^v = \pibarinv(v)(Y^+_{t\kappa_+(v)} - Y^-_{t\kappa_-(v)})$ and
\be\label{kappm}\kappa_\pm(v) = (\pibar(\pibarinv(v)-)-v)\frac{\Pi\{\pm \pibarinv(v)\}}{\Pi^{|\cdot|}\{\pibarinv(v)\}} {\bf 1}_{\Pi^{|\cdot|}\{\pibarinv(v)\} \neq 0}.
\ee

From the above analysis, we can write down the characteristic functions of the trimmed processes.
For each $\theta \in \R$ and $v > 0$, define
\begin{align}\label{char_prp}
 \wt \Phi (\theta, v)
: = \rmi \theta \wt \gamma_v - \frac 1 2 \sigma^2\theta^2  
   &+ \int_{|x|<\pibarinv(v)} \left(  e^{\rmi \theta x } -1 - \rmi \theta x {\bf 1}_{|x|\le 1}  \right) \Pi(\rmd x) \nonumber \\
   &+ \kappa_+(v) (e^{\rmi \theta \pibarinv(v)} - 1) + \kappa_-(v) (e^{-\rmi \theta \pibarinv(v)} - 1).
\end{align}
This is the characteristic exponent of $\wt X_1^v + \wt G_1^v$.
Similarly for $r,s$-asymmetrical trimming, define, for each $u, v >0$ and $\theta \in \R$,
\begin{align*} %\label{char_prp2}
 \Phi(\theta, u, v)
 := \rmi \theta \gamma_{u,v} - \frac 1 2 \sigma^2\theta^2  
   &+ \int_{(-\pibarminv(u),\pibarpinv(v))_*} \left(  e^{\rmi \theta x } -1 - \rmi \theta x {\bf 1}_{|x|\le 1}  \right) \Pi(\rmd x)\nonumber \\
   &+ \rho_+(v) (e^{\rmi \theta \pibarpinv(v)} - 1) + \rho_-(u) (e^{-\rmi \theta \pibarminv(u)} - 1),
\end{align*} which is the characteristic exponent of $X_1^{u,v} + G_1^{+,v} - G_1^{-,u}.$
 
Then the characteristic functions of the trimmed processes are
\begin{align}\label{char_r}
& E \left(e^{\rmi \theta  {}^{(r)} \wt X_t}\right)
=   \int_{(0,\infty)} \exp(t\wt \Phi (\theta, v)) P(\Gamma_r \in t \rmd v) \\
\quad \text{and} \quad \nonumber 
& E \left(e^{\rmi \theta  {}^{(r,s)} X_t}\right)
=  \int_0^\infty \int_{0}^\infty \exp(t\Phi (\theta,u, v)) P(\wt\Gamma_s \in t \rmd u) P(\Gamma_r \in t \rmd v) .
\end{align}

\subsection{Normed and Centered Truncation}

Suppose for some non-stochastic functions $a_t \in \R$ and $b_t > 0$ and a sequence $t_n \dto 0$, a L\'evy process $X_t$ has a limit in distribution, i.e.
\begin{equation}\label{inflim}
 \frac{X_{t_n}-a_{t_n}}{b_{t_n}} \todr Y, \quad \text{as} \quad t \dto 0,
\end{equation} for some a.s. finite nondegenerate random variable $Y$. By Lemma 4.1 in Maller and Mason \cite{MM08}, $Y$ has to be infinitely divisible, say with triplet $(\beta, \tau^2, \Lambda)$.
 In this section, we investigate the convergence of the truncated processes $X_t^{u/t,v/t}$  and $\wt X_t^{v/t}$ with the same centering and norming, for appropriate $u, v > 0$ through the sequence $t_n$. However, in order to relate to the trimmed process, we need to consider not just the truncated processes but also the Poisson number of ties at each truncation level.  With this restriction, we only get convergence through a subsequence in general. Nonetheless, this suffices for our purposes.

For each $t>0$, $u, v>0$,  and $a_t \in \R$, $b_t > 0$ non-stochastic functions, abbreviate the normed, centred, truncated processes including the Poisson number of ties by 
\be\label{def2}
 Z_t^{u,v} : = \frac{X_t^{u/t, v/t} + G_t^{+, v/t}- G_t^{-,u/t} - a_t}{b_t} 
 \quad \text{and} \quad  \wt Z_t^{v} : = \frac{\wt X_t^{v/t} + \wt G_t^{v/t} - a_t}{b_t}.
\ee 
If $(X_{t_n}-a_{t_n})/b_{t_n}$ converges as in \eqref{inflim}, we show that $Z_t^{u,v}$ and $\wt Z_t^v$ also have infinitely divisible limits  at least through a subsequence of $t_n$. Let $\lambar$ and $\lambar^\pm$ denote the tails of the L\'evy measure $\Lambda$ of $Y$.

\begin{lemma}\label{cut_tie}
Suppose $\pibar(0+)=\infty$ and for some non-stochastic functions $a_t$ and $b_t >0$, and sequence $t_n \dto 0$,
\ben\label{inf.1}
\frac{X_{t_n} - a_{t_n}}{b_{t_n}} \todr Y, \quad \text{as} \quad n \to \infty
\een for some a.s. finite infinitely divisible distribution $Y$ with characteristic triplet $(\beta, \tau^2, \Lambda)$. Suppose further that $\Lambda \neq 0$ so there exists $l > 0$ such that $m:= \lambar(l) >0$. Then the following hold.

{\rm (i)} For each continuity point $v$ of $\laminv$ such that $v \in (0,m)$,  $(\wt X_{t_n}^{v/{t_n}} - a_{t_n})/b_{t_n}$ converges in distribution to an infinitely divisible random variable
$\wt Y^v$ as $n \to \infty$, where $\wt Y^v$ is the value at time 1 of a L\'evy process  with canonical triplet  $(\wt\beta_v, \wt\tau^2_v, \wt\Lambda_v)$ given by
\ben\label{cut_trp.2}
\wt \beta_v=  \beta-{\bf 1}_{\{\laminv(v)\le 1\}}\int_{\laminv(v)\le |y|\le 1} y \Lambda(\rmd y),\ \wt \tau_v^2 = \tau^2, \
\wt\Lambda_v(\rmd x)= \Lambda(\rmd x){\bf 1}_{\{|x|<\laminv(v)\}}.
\een

Similarly, for each continuity point $u >0$ of $\lamminv(\cdot)$ and each continuity point $v >0 $ of $\lampinv(\cdot)$, such that $u, v \in (0,m)$, we have
\[ \frac{X_{t_n}^{u/{t_n}, v/{t_n}} - a_{t_n}}{b_{t_n}} \todr Y^{u,v}  \quad \text{as }\quad n \to \infty
\]
where $Y^{u,v}$ has canonical triplet $(\beta_{u,v}, \tau^2_{u,v}, \Lambda_{u,v})$ given by
\begin{align*}\label{cut_trp3}
&\beta_{u,v}=  \beta-{\bf 1}_{\{\lampinv(v)\le 1\}}\int_{\lampinv(v)\le y\le 1} y \Lambda(\rmd y) +{\bf 1}_{\{\lamminv(u)\le 1\}}\int_{\lamminv(u)\le y\le 1} y \Lambda^-(\rmd y)  ,\ \nonumber \\
 &  \tau_{u,v}^2 = \tau^2, \quad \text{and}  \quad
\Lambda_{u,v}(\rmd x)= \Lambda(\rmd x){\bf 1}_{\{ -\lamminv(u)<x<\lampinv(v)\}_*}.
\end{align*}

{\rm (ii)} 
There exists a subsequence $\{t_{n_k} \dto 0\}$ and some infinitely divisible random variables $Y^{u,v}$ and $\wt Y^v$ which may depend on the choice of subsequence such that
\ben\label{inf.2}
Z_{t_{n_k}}^{u,v}\todr Y^{u,v} \quad \text{and}\quad \wt Z_{t_{n_k}}^v \todr \wt Y^v \quad \text{as } k \to \infty,
\een 
for each $u, v \in (0,m)$ that are continuity points of $\lamminv$ and $\lampinv$ respectively, 

In both {\rm (i)} and {\rm (ii)}, the supports of the L\'evy measures of $Y^{u,v}$ and $\wt Y^v$ include the sets $ (-\lamminv(u), \lampinv(v))_*$ and $(-\laminv(v), \laminv(v))_*$ respectively.

\end{lemma}

\begin{proof}
Assume  $\pibar(0+)=\infty$. We prove the case with modulus truncation and to ease the notation we will write $t$ for $t_n$.
We thus assume $ (X_t-a_t)/b_t$ converges as $t\dto 0$ but make no assumption regarding the limit distribution other than that it is a.s. finite. By Kallenberg's conditions (Theorem 15.14, Kallenberg \cite{kal02}),
the following limits hold for each continuity point $x > 0$ of $\lambar^\pm(\cdot)$:
\begin{align}\label{cut_KL}
\lim_{t\dto 0} t\pibar^\pm(x b_t)=\lambar^\pm(x), \
\lim_{t \dto 0}\frac{t V(x b_t)}{b_t^2} = \tau^2 + \int_{|y|\le x} y^2 \Lambda(\rmd x), \
\lim_{t \dto 0} \frac{t\nu(b_t)-a_t}{b_t} = \beta.
\end{align}
By properties of inverse monotone functions (Proposition 0.1 in Resnick p.5  \cite{res87}), the first relation in \eqref{cut_KL} implies that 
$\pibarinv(v/t)/b_t \to \laminv(v)$ for each continuity point $v>0$ of $\laminv$.
By  \eqref{char_prp} and \eqref{char_r},  we have 
\begin{align}\label{inf.3}
E\left(\exp(\rmi \theta \wt Z_t^v)\right) 
	&= \exp\Bigg\{ 
	  \rmi \theta \left(\frac{t\wt \gamma_{v/t} - a_t}{b_t}-t\int_{b_t\le |x| \le 1, |x|<\pibarinv(v/t)} \frac{x}{b_t} \, \Pi(\rmd x)\right) \nonumber \\
	 & -\frac 1 2 \frac{t\sigma^2\theta^2}{b_t^2}+ t\int_{|x|<\pibarinv(v/t)} \left(e^{\rmi \theta x/b_t } -1 -\rmi\theta x/b_t {\bf 1}_{|x|\le b_t}\right) \, \Pi( \rmd x)\nonumber \\
	& + t\kappa_+(v/t)\left(e^{\rmi\theta \pibarinv(v/t)/b_t}-1\right) +t\kappa_-(v/t)\left(e^{-\rmi\theta \pibarinv(v/t)/b_t}-1\right) 	
		\Bigg\}.  
\end{align}
 By \eqref{trip1}, the resulting centering, i.e. the first line on the RHS of \eqref{inf.3}, equals 
\begin{align}\label{inf.4}
   & \left(\frac{t\gamma-a_t}{b_t} - {\bf 1}_{\{\pibarinv(v/t)\le 1\}} t \int_{\pibarinv(v/t)\le |x|\le 1 } \frac{x}{b_t} \Pi(\rmd x)
- t \int_{ b_t < |x|\le 1, |x|<\pibarinv(v/t)} \frac{x}{b_t} \, \Pi(\rmd x )  \right)\nonumber \\
& =   \left( \frac{t\gamma-a_t}{b_t} -{\bf 1}_{\{\pibarinv(v/t)\le b_t \}} t\int_{\pibarinv(v/t)\le |x|\le b_t} \frac{x}{b_t} \, \Pi( \rmd x ) - t\int_{b_t < |x| \le 1 } \frac{x}{b_t} \Pi(\rmd x)  \right)
 \nonumber \\
 & = \frac{t\nu(b_t) - a_t}{b_t} - {\bf 1}_{\{\pibarinv(v/t)/b_t\le 1\}} \int_{\pibarinv(v/t)/b_t\le |x| \le 1} x \, t\Pi(b_t \rmd x ) \nonumber \\
 &\stackrel{t \dto 0}{ \longrightarrow} 
  \beta-{\bf 1}_{\{\laminv(v)\le 1\}} \int_{\laminv(v)\le |x|\le 1} x \, \Lambda(\rmd x):= \wt \beta_v.
\end{align} In the last line of \eqref{inf.4}, note that $\laminv(v) > 0$ for $v \in (0,m)$ which is a continuity point of $\lambar$, hence making use of \eqref{cut_KL} and dominated convergence, we arrive at the limit $\wt \beta_v$.

By assuming $\Lambda \neq 0 $, for each $v \in (0, m)$ a continuity point of $\lambar$, where $m = \lambar(l) > 0$ for some $l >0$, we have $\pibarinv(v/t)/b_t \to \laminv(v) \ge \laminv(m)\ge l >0$.  So $\veps b_t < \pibarinv(v/t)$ for all $0<\veps < \min(l,1) $, $v \in (0,m)$ and all sufficiently small $t$.
Hence we can break up the second line in \eqref{inf.3} into two parts. 
First consider the integral on $ \{|x| \le \veps b_t \}$:
\begin{align}\label{inf.4a}
& -\frac{t\sigma^2 \theta^2}{2b_t^2} + t \int_{|x|\le \veps b_t} \left( e^{\rmi \theta x /b_t} -1 -  \rmi \theta x /b_t  \right) \Pi(\rmd x) \nonumber \\
& =-\frac{t\sigma^2 \theta^2}{2b_t^2} + 
t \int_{|x|\le \veps b_t} \left( \frac{(i\theta x)^2 }{2 b_t^2}  + O\left(\frac{|x^3|}{b_t^3}\right) \right) \Pi(\rmd x) \nonumber \\ 
& = -\frac {t\theta^2}{2b_t^2} \left(\sigma^2 + \int_{|x| \le \veps b_t} x^2 \Pi(\rmd x)\right) + t \int_{|x|\le \veps b_t} O\left(\frac{|x|^3}{b_t^3}\right)\Pi(\rmd x) \nonumber\\
& = -\frac{t\theta^2 V(\veps b_t)}{2 b^2_t} + O\left( \frac{\veps t V(\veps b_t)}{b^2_t} \right).
\end{align} By \eqref{cut_KL}, 
\[ \lim_{\veps \dto 0} \lim_{t \dto 0} \frac{tV(\veps b_t)}{b_t^2} = \tau^2 .
\]
The second term in \eqref{inf.4a} is $O(\veps)$ as $t \dto 0$ hence arbitrarily small. So the expression in \eqref{inf.4a} tends to $-\theta^2 \tau^2/2$ as $t \dto 0$ then $\veps \dto 0$.

Next consider the component of the integral in the second line of \eqref{inf.3} on $\{\veps b_t <|x| < \pibarinv(v/t)\}$:  
\begin{align*}\label{inf.4c}
& t \int_{\veps b_t <|x| < \pibarinv(v/t)} \left( e^{\rmi \theta x /b_t} -1 -  \rmi \theta x /b_t {\bf 1}_{|x| \le b_t}  \right) \Pi(\rmd x) \nonumber \\
& = t \int_{\veps < |x| < \pibarinv(v/t)/b_t} \left( e^{\rmi \theta x} -1 - \rmi \theta x {\bf 1}_{|x| \le 1} \right) \Pi(  b_t \rmd x)  \nonumber \\
&\to \int_{|x| < \laminv(v)} \left( e^{\rmi \theta x} -1 - \rmi \theta x {\bf 1}_{|x| \le 1} \right)  \Lambda(\rmd x) \quad \text{as} \quad t \dto 0 \text{ and then } \veps \to 0.
\end{align*}
Therefore the overall limit as $t \dto 0$ for the second line in \eqref{inf.3} is 
\be\label{inf.4bc} 
 -\frac 1 2 \theta^2 \tau^2 + \int_{|x|<\laminv(v)} \left(e^{\rmi \theta x} -1 -\rmi\theta x{\bf 1}_{|x|\le 1}\right) \, \Lambda ( \rmd x).
\ee 
From here we can see that the support of the limit L\'evy measure is $\{ |x| < \laminv(v)\}_*$ without considering the ties component. The ties component, if present, will only enlarge the support by including one or both boundary points. This proves Part {\rm (i)}, for the convergence of $(\wt X_{t_n}^{v/{t_n}}-a_{t_n})/b_{t_n}$.

For Part {\rm (ii)}, the Poisson number of ties are added to $\wt Z_t^v$ in \eqref{def2}. This corresponds to the last line of \eqref{inf.3} in the characteristic function.
As before, we fix $v \in (0,m)$ to be a continuity point of $\laminv$. By \eqref{kappm}, the ties disappear if $\pibarinv(v/t)$ is not an atom of $\Pi^{|\cdot|}$. Let $\{t_n\} \dto 0$ be the given sequence.  If there exists a subsequence $\{t_{n_k}\} \dto 0$ such that $\pibarinv(v/t_{n_k})$ is a continuity point of $\Pi$ for all $\{t_{n_k}\}$ for sufficiently large $k$, then the ties components converge to $0$ as $k \to \infty$, and Part {\rm (ii)} of the Lemma is true for this subsequence.
 
Suppose this is not the case. In this situation we have to choose a further subsequence. Henceforth without loss of generality, we assume additionally that $\Pi^{|\cdot|}\{\pibarinv(v/t_n)\} \neq 0$ for all $n \in \N$. 
Observe from \eqref{def2} that
\begin{equation}\label{BB1}
\wt Z_{t_n}^{v}  = \frac{\wt X_{t_n}^{v/{t_n}}-a_{t_n}}{b_{t_n}} + \frac{\wt G_{t_n}^{v/{t_n}}}{b_{t_n}}.
\end{equation} Since it is shown in Part {\rm (i)} that the first term in \eqref{BB1} converges to an infinitely divisible random variable with characteristic triplet $(\wt \beta_v, \wt \tau_v^2, \wt \Lambda_v)$,  we only need to show that $\wt G_{t_n}^{v/{t_n}}/b_{t_n}$ has a limit through a subsequence. Recall from \eqref{2rrep_1}-\eqref{kappm}, 
\[ \frac{\wt G_{t_n}^{v/{t_n}}}{b_{t_n} } = \frac{\pibarinv(v/t_n)}{b_{t_n}} \left(Y^+_{t_n \kappa_+(v/t_n)} - Y^-_{t_n \kappa_-(v/t_n)}\right)
\] where $Y^\pm $ are Poisson processes with unit mean, independent of $\wt X^{v/t_n}_{t_n}$.
By \eqref{kappm},
\begin{align}\label{inf.5}
t\kappa_\pm(v/t)
= t\left(\pibar(\pibarinv(v/t)-)- \frac{v}{t}\right)\frac{\Pi\{\pm \pibarinv(v/t)\}}{\Pi^{|\cdot|}\{\pibarinv(v/t)\}}
= \int_{v}^{t\pibar(\pibarinv(v/t)-)} g^\pm(\pibarinv(v/t)) \rmd u,
\end{align} where $g^\pm = \rmd \Pi^\pm /\rmd \Pi^{|\cdot|}$ are the Radon-Nikodym derivatives of $\Pi^\pm$ with respect to $\Pi^{|\cdot|}$. Since $\pibarinv(v/t)$ is an atom of $\Pi^{|\cdot|}$, 
\[g^\pm (\pibarinv(v/t)) = \frac{\Pi\{\pm \pibarinv(v/t)\}}{\Pi^{|\cdot|}\{\pibarinv(v/t)\}}.
\]
For each $w > 0$, $t>0$, define
\begin{equation*}\label{lam.def}
\lambda_t^\pm (w) = \int_0^w g^\pm(\pibarinv(u/t))\rmd u.
\end{equation*}
Note that $\pibarinv(z/t) = \pibarinv(v/t)$ for each $z \in (v, t\pibar(\pibarinv(v/t)-))$.
Hence, \eqref{inf.5} equals
\ben\label{inf.6}
t\kappa_\pm(v/t) = \lambda_t^\pm(t\pibar(\pibarinv(v/t)-))- \lambda_t^\pm(v).
\een Observe that $\lambda_t^\pm(t\pibar(\pibarinv(v/t)-))$ and $\lambda_t^\pm(v)$ are nondecreasing in $v$. Therefore by Helly's selection theorem, there exists a subsequence $\{t_{n_k} \dto 0\}$  of $\{t_n\}$ and nondecreasing functions $h^\pm(\cdot)$ and $l^\pm (\cdot)$ such that
\begin{equation}\label{inf.7}
 \lambda_{t_{n_k}}^\pm({t_{n_k}}\pibar(\pibarinv(v/{t_{n_k}})-)) \to h^\pm(v)\qquad \text{and } \qquad \lambda_{t_{n_k}}^\pm(v)\to l^\pm(v) \quad \text{as} \quad k \to \infty.
\end{equation}
Therefore $ 0 <t_{n_k}\kappa_\pm(v/t_{n_k}) \to h^\pm(v)-l^\pm(v) =: \lambda^\pm(v)$. We claim that these quantities are finite for each $v \in (0,m)$. To see this,
note that for each $v \in (0,m)$, $\laminv(v) \ge l>0$.  Hence there exists a $\delta  > 0$ such that $c_v : = \laminv(v) - \delta > 0$. Since $\pibarinv(v/t)/b_t \to \laminv(v)$, thus $\pibarinv(v/t)\ge b_t c_v$, for all sufficiently small $t$. Hence
\begin{equation*}\label{inf.8}
t\pibar(\pibarinv(v/t)-) \le t\pibar(b_t c_v) \to \lambar ( c_v) <\infty.
\end{equation*}
This shows that for each $v\in (0,m)$, $t\kappa_\pm(v/t) <\infty$ for all sufficiently small $t > 0$.
To summarise, by \eqref{inf.4}, \eqref{inf.4bc} and \eqref{inf.7}, $E\left(\exp(\rmi \theta \wt Z_{t_{n_k}}^v)\right)$ tends, as $k\to \infty$, to
\begin{align}\label{inf.9}
&\exp\Bigg\{ 
				\rmi \theta \wt \beta_v  - \frac 1 2 \theta^2 \tau^2 + \int_{|x|<\laminv(v)} \left(e^{\rmi \theta x} - 1 -\rmi \theta x {\bf 1}_{|x|\le 1}\right) \Pi(\rmd x) \nonumber \\
			&\hspace*{1.5in} + \lambda^+(v)\left(e^{\rmi \theta \laminv(v)} -1\right)
			+ \lambda^-(v)\left(e^{-\rmi \theta \laminv(v)} -1 \right)	\Bigg\} : = \wt \psi_v(\theta).
\end{align}
Note that \eqref{inf.9} is the characteristic function of the limit random variable, say $\wt Y^v$, which is a convolution of an infinitely divisible random variable with canonical triplet $(\wt \beta_v, \tau^2, \wt \Lambda_v)$ and two independent Poisson numbers at $\pm\laminv(v)$ respectively.

This completes the proof of the modulus truncation. Asymmetrical truncation can be computed analogously.
\end{proof}

\section{Inequalities}\label{sect:ineq}
In this section, we derive inequalities that relate the tails of trimmed processes with their largest jumps.
First let us write out the marginal distribution of the $(r+1)^{st}$ ordered jump from the representation formulae in \eqref{rand_dis3} and \eqref{2rrep_1}.
\begin{lemma}\label{rJ}
Let $y >0$. Then
% tail of marginal distribution of the $(r+1)^{st}$ largest jump in modulus is
\begin{align}\label{rJ0}
P(|\wt{\Delta X}_t^{(r+1)} | > y ) &= \int_{0}^{t\pibar(y)}P(\Gamma_{r+1} \in \rmd v) \nonumber \\
&= \int_{0}^{t\pibar(y)}P(\Gamma_{r} \in \rmd v)-e^{-t\pibar(y)}\frac{(t\pibar(y))^{r}}{r!}.
\end{align}
%Denote the $(r+1)^{st}$ largest positive and negative jumps in magnitude by $\Delta X_t^{(r+1), \pm}$ respectively. Then
Similarly,
\begin{align}\label{rJ1}
P(\Delta X_t^{(r+1), \pm}  > y ) 
&= \int_{0}^{t\pibar^\pm(y)}P(\Gamma_{r+1} \in \rmd v) \nonumber \\
&= \int_{0}^{t\pibar^\pm(y)}P(\Gamma_{r} \in \rmd v)-e^{-t\pibar^\pm(y)}\frac{(t\pibar^\pm(y))^{r}}{r!}.
\end{align}
Hence,
\be\label{rJ3}
e^{-t\pibar(y)} \frac{(t\pibar(y))^{r+1}}{(r+1)!} \le P(|\wt{\Delta X}_t^{(r+1)} | > y ) \le \frac{(t\pibar(y))^{r+1}}{(r+1)!},
\ee
and 
\be\label{rJ2}
e^{-t\pibar^\pm(y)} \frac{(t\pibar^\pm(y))^{r+1}}{(r+1)!} \le P({\Delta X}_t^{(r+1),\pm}  > y ) \le \frac{(t\pibar^\pm(y))^{r+1}}{(r+1)!}.
\ee

\end{lemma}

\begin{proof} 
From the representation in \eqref{2rrep_1},
\ben\label{rJ5}
P(|\wt{\Delta X}_t^{(r+1)} | > y ) = P(\pibarinv(\Gamma_{r+1}/t)> y) = P(\Gamma_{r+1} < t\pibar(y)).
\een
This gives the first identity in \eqref{rJ0}. Integrate by parts to get
\ben\label{rJ6}
\int_{0}^{t\pibar(y)} \frac 1 {r!} x^r e^{-x} \rmd x  = \frac 1 {r!}\left(-(t\pibar(y))^re^{-t\pibar(y)}
+ \int_0^{t\pibar(y)} r x^{r-1}e^{-x}\rmd x   \right).
\een
Then we can read off the second identity in \eqref{rJ0}. \eqref{rJ1} can be proved similarly. The inequality in \eqref{rJ3} is straightforward by observing that 
\[ e^{-t\pibar(y)} \int_0^{t\pibar(y)} \frac{x^{r}}{r!}\rmd x \le
\int_0^{t\pibar(y)} e^{-x} \frac{x^{r}}{r!}\rmd x \le \int_0^{t\pibar(y)} \frac{x^{r}}{r!}\rmd x.
\]
\eqref{rJ2} can be proved similarly.
\end{proof}

\begin{remark}\label{cdfdis1}
The tail of the cumulative distribution function (cdf) of the modulus ordered jumps satisfies
\ben\label{cdfdis2}
P(|\wt{\Delta X}_t^{(r)}| > y ) = P(\Gamma_r < t\pibar(y))\quad 
\text{and} \quad P(|\wt{\Delta X}_t^{(r)}| \ge  y ) = P(\Gamma_r < t\pibar(y-)).
\een Therefore the discontinuity points of the distribution of ordered jumps coincide with the atoms of its L\'evy measure $\Pi^{|\cdot|}$, which are at most countable. 
\end{remark}

We state our main inequality relating the cdf of trimmed processes with that of the normed ordered jumps. A version of the following inequality appeared in Buchmann et al. \cite{bfm14} in which only the one-sided maximal trimmed process is considered.

\begin{proposition}\label{keyIn}
Assume $\pibar(0+)=\infty$. For each $t, x >0$, $r,s \in \N$,
%  where $4x$ is a continuity point of $\pibar$ or $\pibar^\pm$ respectively for \eqref{keyin0} or \eqref{key0.+}, \eqref{key0.-},
let $a_t \in \R$ be any non-stochastic function. We have
\be\label{key0.+}
4P(|{}^{(r,s)} X_t -a_t| > x) \ge \max\left(P({\Delta X}_t^{(r+1)}> 4x),  P( {\Delta X}_t^{(s+1),-} > 4x)\right).
\ee
By letting $r = 0$ or $s = 0$, we get 
\ben\label{key0.one}
4P(|{}^{(r)} X_t -a_t| > x) \ge P({\Delta X}_t^{(r+1)}> 4x) 
\een and
\ben\label{key0.one:1}
4P(|{}^{(s,-)}X_t -a_t | > x) \ge P( {\Delta X}_t^{(s+1),-} >4x).
\een
Similarly the modulus trimmed process satisfies
\be\label{keyin0}
4P(|{}^{(r)}\wt X_t-a_t| > x) \ge P\big(\big|\wt {\Delta X}_t^{(r+1)}\big|> 4x\big).
\ee
\end{proposition}

\begin{proof}
We first prove \eqref{key0.+}. Assume $\pibar(0+) = \infty$. 
%Suppose that $4x > 0$ is a continuity point of $\pibar$. 
By the representation in \eqref{rand_dis3}, 
\begin{align}\label{keyin:1}
& P\left( \left| {}^{(r,s)}X_t -a_t\right| > x \right) \nonumber \\
=&\int_{u,v} P\left( \left|X^{ u,v}_t+ G_t^{+,v} - G_t^{-,u} - a_t\right| >  x  \right)P(\Gamma_r \in t\rmd v, \wt \Gamma_s \in t\rmd u). 
\end{align}
Define $W^{u,v}_t : = X^{ u,v}_t+ G_t^{+,v} - G_t^{-,u}$. Recall that $\rho_\pm(w) = \pibar^\pm(\pibarpminv(w)-) - w$. 
Then the aggregate L\'evy measure of $(W_t^{u,v})$ is 
\[ \Theta(\rmd x) : = \Pi(\rmd x ){\bf 1}_{-\pibarpinv(u) < x < \pibarpinv(v)} + \rho_+ (v)\delta_{\{\pibarpinv(v)\}} + \rho_- (u) \delta_{\{-\pibarpinv(u)\}}.
\]
 Denote the tails of L\'evy measure $\Theta(\cdot)$ by $ \thetabar^\pm$. Then for $0<x <\pibarpinv(v)$,
 \[ \thetabar^+(x) = \pibar^+(x) - \pibar^+(\pibarpinv(v)-) + \rho_+(v) = \pibar^+(x) - v.
 \] Similarly, for $0<x < \pibarminv(u)$,
 we have 
 \[ \thetabar^-(x) = \pibar^-(x) - u.
 \]
Let $(\bar X_t^{u,v})$, $(\bar G^{\pm, w}_t)$ be processes independent of $(X_t^{u,v})$ and $(G_t^{\pm,w})$ respectively but with the same law. Define the symmetrised process
\[ \wh W_t^{u,v} := X^{ u,v}_t+ G_t^{+,v} - G_t^{-,u}- \left(\bar X^{ u,v}_t+ \bar G_t^{+,v} - \bar G_t^{-,u}\right).
\]
Then the symmetrised process $\wh W_t^{u,v}$ has L\'evy measure $W(\rmd x) = \Theta(\rmd x) + \Theta(- \rmd x)$. 
%Then the tails of L\'evy measure $W(\cdot)$ are
%\begin{align}
% &\overline{W}^\pm(x)
% = \thetabar^+(x) + \thetabar^-(x) \nonumber \\
% &= \begin{cases}
% \pibar(x) - v- u, & x < \pibarminv(u)\wedge \pibarpinv(v). \nonumber \\
% \pibar^+(x) - v \quad  \text{or} \quad \pibar^-(x)- u, & \pibarminv(u)\wedge \pibarpinv(v) \le x \le \pibarminv(u)\vee \pibarpinv(v). \nonumber \\
% 0, & x > \pibarminv(u)\vee \pibarpinv(v).
% \end{cases}
%\end{align}

By the symmetrisation inequality (Lemma 1 in Feller \cite{feller66} p.147) and L\'evy's maximal inequality (Lemma 2 in Feller \cite{feller66} p.147, also see Lemma 1.1 in Fan \cite{fan_thesis}), 
\begin{align}\label{new1}
P\left( \left|X^{ u,v}_t+ G_t^{+,v} - G_t^{-,u} - a_t\right| >  x  \right)
&\ge \frac 1 2 P(|\wh W_t^{u,v}| > 2x) \nonumber \\
&\ge \frac 1 4 P(\sup_{0<s\le t} |\Delta \wh W_s^{u,v}| > 4x)
\end{align}
where $(\Delta \wh W_s^{u,v})_{0<s\le t}$ is the jump process of $\wh W_t^{u,v}$. Denote the positive and negative jump processes of $\wh W_t^{u,v}$ by $(\Delta \wh W_s^{+,u,v})_{0<s\le t}$ and $(\Delta \wh W_s^{-,u,v})_{0<s\le t}$ respectively. Note that they are independent L\'evy processes with the same law. Each of $(\Delta \wh W_s^{+,u,v})_{0<s\le t}$ and $(\Delta \wh W_s^{-,u,v})_{0<s\le t}$ has L\'evy tails $\overline{W}^\pm(\cdot) = \thetabar^+(\cdot) + \thetabar^-(\cdot) = \thetabar(\cdot)$. Hence,
\begin{align}\label{new2}
P\left(\sup_{0<s\le t} \Delta \wh W_s^{+, u,v} > 4x \right) &= P\left(\sup_{0<s\le t} \Delta \wh W_s^{-, u,v} > 4x \right) \nonumber \\
    & = 1-e^{-t\thetabar(4x)} \nonumber \\
    & \ge (1-e^{-t\thetabar^+(4x)}) \vee (1-e^{-t\thetabar^-(4x)}) \nonumber \\
	&= P\left(\sup_{0<s\le t} \Delta W_s^{+,v} > 4x \right) \vee P\left(\sup_{0<s\le t} \Delta W_s^{-,u} > 4x \right),
\end{align} where $(\Delta W_s^{+,v})_{0<s\le t}$ and $(\Delta W_s^{-,u})_{0<s\le t}$ are positive and negative jump processes of $W_t^{u,v}$ and they are independent of each other.
By \eqref{new1} and \eqref{new2}, 
\begin{align*}
 P(\sup_{0<s\le t} |\Delta \wh W_s^{u,v}| > 4x) &\ge P\left(\sup_{0<s\le t} \Delta \wh W_s^{+, u,v} > 4x \right) \nonumber \\
							&\ge  P(\sup_{0<s\le t} \Delta W_s^{+,v} > 4x) \vee P(\sup_{0<s\le t} \Delta W_s^{-,u} > 4x).
\end{align*}
On the set $\{ v < \pibar^+ (4x)\} = \{4x < \pibarpinv(v)\}$, by \eqref{keyin:1} and \eqref{new1}, 
\begin{align}\label{keyin:3}
& 4 P(|^{(r,s)}X_t -a_t| > x ) \nonumber \\
&\ge \int_{0}^{\pibar^+(4x)}\int_{u \in (0,\infty)} P(\sup_{0\le s\le t} \Delta W_s^{+,v} > 4x )P(\Gamma_r \in t\rmd v, \wt \Gamma_s \in t\rmd u) \nonumber \\
&=  \int_{0}^{t\pibar^+(4x)}\int_{u \in (0,\infty)}\left(1- e^{-t(\pibar^+(4x)-v/t)}\right) P(\Gamma_r \in \rmd v, \wt \Gamma_s \in \rmd u)  \nonumber \\
&\ge \int_{0}^{t\pibar^+(4x)} \left(1- e^{-t\pibar^+(4x) + v } \right) P(\Gamma_r \in \rmd v).
\end{align}
The last line of \eqref{keyin:3} is, by \eqref{rJ1}, 
\begin{equation*}\label{keyin:3.5}
  \int_{0}^{t\pibar^+ (4x)} P(\Gamma_r \in \rmd v) -  e^{-t \pibar^+(4x)} \int_{0}^{t\pibar^+(4x)} \frac{v^{r-1}}{(r-1)!} \rmd v = P(\Delta X_t^{(r+1)} > 4x) .
\end{equation*}

Similarly, if we consider the set $\{ u < \pibar^-(4x) \} = \{ 4x < \pibarminv(u) \}$ and replace the integrand in the second line of \eqref{keyin:3} by $ P(\sup_{0\le s\le t} \Delta W_s^{-,u} > 4x )$. The exponent in the third line of \eqref{keyin:3} becomes $-t(\pibar^-(4x) - u/t)$. This leads to $P(\Delta X_t^{(s+1),-} > 4x)$. Hence we complete the proof for \eqref{key0.+}. 
\eqref{keyin0} is proved similarly.

\end{proof}

\section{Proof of Theorems}\label{sect:prf}

\begin{proof}[Proof of Theorem~ \ref{tighteqv} {\rm (a)}:] 
Take $r , s \in \N$. Let $({}^{(r,s)}S_t)$ be tight. 
For each $x > 0$ and $t > 0$, by \eqref{key0.+} in Proposition \ref{keyIn},   
\ben\label{tiger1}
4P\left( \left|{}^{(r,s)}  X_t - a_t\right| >  x b_t\right) \ge \max\left( P\left( {\Delta X}_t^{(r+1)} > 4 x b_t\right),  P\left( {\Delta X}_t^{(s+1),-} > 4 x b_t\right) \right).
\een 
Take $\limsup_{t \dto 0}$ and then $\lim_{x\to \infty}$ to obtain
\begin{align*}\label{tiger2}
0 &=\lim_{x\to \infty}\limsup_{t  \dto 0} 4P\left(\left|{}^{(r,s)}S_t \right| > x \right)  \nonumber \\
   &\ge  \lim_{x\to \infty}\limsup_{t \dto 0}  \max\left( P\left({\Delta X}_{t}^{(r+1)}/b_t > 4 x  \right),  P\left({\Delta X}_{t}^{(s+1),-}/b_t > 4 x  \right) \right).
\end{align*}
This implies that $\left( {\Delta X}_t^{(r+1)}/b_t \right) $ and $\left(  {\Delta X}_t^{(s+1),-}/b_t\right)$  are tight families as $t \dto 0$. 
Hence there exists $x_0 > 0$ such that 
\be\label{tiger2:a}
\limsup_{t\dto 0} P(\Delta X_t^{(r+1)}/b_t > x) \le 1/2 \quad \text{for all } x > x_0.
\ee
For an $x > x_0$, suppose there exists a sequence $\{t_k \} \dto 0$ such that 
$t_k \pibar^+(b_{t_k} x) \to \infty$ as $k \to \infty$. 
Then by \eqref{rJ1},
\be\label{t13}
P(\Delta X_{t_k}^{(r+1)} > b_{t_k} x ) = \int_0^{t_k\pibar^+(b_{t_k}x)} P(\Gamma_{r+1} \in \rmd v) \to 1 \quad \text{as} \quad k \to \infty,
\ee
which contradicts \eqref{tiger2:a}. Therefore  $\limsup_{t\dto 0} t\pibar^+( b_t x) < \infty$
 for each $x >x_0$. By \eqref{rJ2},
\be\label{t12}
0 = \lim_{x\to \infty} \limsup_{t \dto 0} P\left(\Delta X_t^{(r+1)} > b_t x\right) \ge \lim_{x\to \infty} \limsup_{t \dto 0} e^{-t\pibar^+( b_t x)} \frac{(t\pibar^+( b_t x))^{r+1}}{(r+1)!}.
\ee
So we must have that $\lim_{x\to \infty} \limsup_{t\dto 0} t\pibar^+(b_t x) = 0$. By the same reasoning we also have $\lim_{x\to \infty} \limsup_{t\dto 0} t\pibar^-(b_t x) = 0$.

Conversely, assume $\lim_{x \to \infty} \limsup_{t \dto 0} t\pibar^+(b_t x) = 0 $. By \eqref{rJ2}, for each $r \in \N$, $x > 0$,
\[ \lim_{x \to \infty} \limsup_{t \dto 0}P(\Delta X_t^{(r)} > x b_t) \le \lim_{x \to \infty} \limsup_{t \dto 0} \frac{(t\pibar^+( b_tx))^r}{r!} = 0
\] This proves statements \eqref{ordersrp} and \eqref{ordersrm}.

Recall that the sum of tight families is again a tight family. Since 
$ {}^{(s,-)}S_t  = {}^{(r,s)}S_t + \sum_{i=1}^r \Delta X_t^{(i)}/b_t $, ${}^{(s,-)}S_t$ is also tight at $0$. 
Similarly, since ${}^{(r)}S_t = {}^{(r,s)}S_t - \sum_{i=1}^s \Delta X_t^{(i),-}$, we conclude that ${}^{(r)}S_t$ is tight at $0$.
Note that $S_t = {}^{(s,-)}S_t - \sum_{i=1}^s \Delta X_t^{(i),-}/b_t$, thus $S_t$ is tight at $0$. This completes the proof of Part {\rm (a)} and Part {\rm (b)} is proved similarly.
\end{proof}

Before proving the main theorem, we write down a useful lemma to eliminate the easy direction.
\begin{lemma}\label{max1}
If there exists a subsequence $t_k \dto 0$ such that $(X_{t_k}-a_{t_k})/b_{t_k} \topr 0$ as $k \to \infty$ or
$(X_{t_k}-a_{t_k})/b_{t_k} \todr N(0,1)$ as $k \to \infty $,
then $\wt{\Delta X}^{(i)}_{t_k}/b_{t_k} \topr 0$ and ${\Delta X}^{(i),\pm}_{t_k}/b_{t_k} \topr 0$ for $i = 1, 2, 3, \dots$ as $k \to \infty$.
\end{lemma}

\begin{proof}
Either convergence implies, by \eqref{cut_KL}, that $t_k\pibar(b_{t_k}x) \to 0$ for all $x>0$, and this implies 
\[ P(|\wt{\Delta X}^{(1)}_{t_k}|/b_{t_k} > \veps) = 1 - e^{-t_k\pibar(b_{t_k}\veps)} \to 0
\]
for any $\veps > 0$. Hence $|\wt{\Delta X}^{(1)}_{t_k}|/b_{t_k} \topr 0$. Thus ${\Delta X}^{(i),\pm}_{t_k}/b_{t_k} \topr 0$ for $i =1 ,2, \ldots $ as $k \to \infty$. 
\end{proof}

\begin{proof}[Proof of Theorem~\ref{convNc}:] \
Necessity follows from Lemma \ref{max1}. We shall prove the sufficiency.
Assume \eqref{convNm+}.
If $\sigma^2 >0$, the truncated second moment function $V(x) \ge \sigma^2 > 0$, thus
\[  \frac{x^2 \pibar(x)}{V(x)} \to 0.
\]By \eqref{DN}, this implies $X_t$ is in the domain of attraction of a normal distribution at 0, in which case \eqref{convNm1} holds with $N(0,\sigma^2)$ on the RHS.
But then $\sigma^2 =1$ since the limit distribution is $N(0,1)$. So we can suppose $\sigma^2 =0 $ in what follows.

First we deal with the degenerate limit. Suppose, without loss of generality, the limit distribution is degenerate at $0$. 
Then the LHS of \eqref{key0.+}, with $x$ replaced by $xb_t$, tends to $0$ as $t \dto 0$, so 
\[ \frac{\Delta X_t^{(r+1)}}{b_t} \topr 0  \quad\text{and} \quad
	 \frac{\Delta X_t^{(s+1),-}}{b_t} \topr 0.
\] By \eqref{rJ2}, this implies, for each $x > 0$,
\[ 0 = \lim_{t\dto 0} P(\Delta X_t^{(r+1),\pm} > x b_t) \ge \lim_{t\dto 0}  e^{-t\pibar^\pm(xb_t)} \frac{(t\pibar^\pm (xb_t))^{r+1}}{(r+1)!}.
\]
By a similar argument as in \eqref{t13}--\eqref{t12}, the degeneracy of $^{(r,s)}S_t$ implies 
\[\limsup_{t\dto 0} t\pibar^\pm(xb_t) <\infty \quad \text{for } x > 0.
\] Therefore as $t\dto 0$, ${\lim_{t\dto 0} t\pibar^\pm(4xb_t)  = 0}$ for all $x > 0 $. As in Lemma \ref{max1}, $\Delta X_t^{(i),\pm}/b_t \to 0$, $i = 1,2, \ldots $. Thus the original normed and centered process also converges, that is
\[ S_t =  {}^{(r,s)}S_{t} + \sum_{i=1}^r \frac{\Delta X_t^{(i)}}{b_t} 
					- \sum_{j=1}^s \frac{\Delta X_t^{(j),-}}{b_t} \topr 0.  
\] This completes the proof for the case with a degenerate limit.

Now we concentrate on the non-trivial case where the limit distribution of $^{(r,s)}S_t$ is $N(0,1)$.
Since ${}^{(r,s)}S_t$ is tight at $0$, by Theorem \ref{tighteqv}, $S_t$ is also tight at $0$, which is equivalent to $S_t$ being relatively compact. Therefore every sequence has a further subsequence convergent in distribution.  In fact, $S_t$ is stochastically compact, i.e.  no subsequence could have a degenerate limit in distribution. If this were not so, there would be a subsequence, say $\{t_k\}$, through which ${}^{(r,s)}S_{t_k}$ converged to a degenerate distribution. By Lemma \ref{max1}, $\Delta X^{(1),\pm}_{t_k}/b_{t_k}$ would tend to $0$ in probability, and so the trimmed process $({}^{(r,s)} X_{t_k}-a_{t_k})/b_{t_k}$ would converge to the same degenerate distribution. But this contradicts the assumption that ${}^{(r,s)}S_t \to N(0,1)$ as $t \dto 0$.

Therefore, for each sequence $\{t_k\}$, there exists a further subsequence (also denoted $\{t_k\}$) such that
$(X_{t_k}-a_{t_k}) /b_{t_k}\todr Z$ as $k \to \infty$ for some a.s. finite nondegenerate infinitely divisible random variable $Z$ with canonical triplet $(\alpha_z, \tau_z^2, \Pi_z)$, say.
For each continuity point $x>0$ of $\Pi_z$, by \eqref{cut_KL},
\ben \label{cov*2}
\lim_{k \to \infty}t_k \pibar(b_{t_k}x) = \pibar_z(x) \quad
 \text{ and }\quad
\lim_{k \to \infty} \frac{t_kV(b_{t_k}x)}{b^2_{t_k}}=\tau_z^2+\int_{ |y|\le x} y^2 \Pi_z(\rmd y).
\een
%%where $\nu(x) =\gamma - \int_{x < |y|\le 1} y\Pi(\rmd y) $ and $V(x) =\sigma^2 + \int_{ %%|y|\le x} y^2 \Pi(\rmd y) $. \\
We will show that $\Pi_z(\cdot) \equiv 0$. Suppose not. 
Then the set 
\[ A := \{x : \pibar_z( x) > 0\} \neq \emptyset.
\]
Let the infimum of $A$ be $l > 0$ and $m = \pibar^+_z (l) \wedge \pibar^-_z(l) > 0$.
By the representation in \eqref{rand_dis3},  for any $x > 0$ and $t > 0$,
\begin{align}\label{cov5}
P(^{(r,s)}S_{t_k} > x)&= \int_{u, v \in (0, \infty)} P(Z_{t_k}^{u,v} > x) P(\Gamma_r \in \rmd v, \wt \Gamma_s \in \rmd u)  \nonumber \\
                &\ge \int_{u, v \in (0, m)} P(Z_{t_k}^{u,v} > x) P(\Gamma_r \in \rmd v, \wt \Gamma_s \in \rmd u) ,
\end{align} where 
\[ Z_t^{u,v} := \frac{X_t^{v/t, u/t}+ G^{+,v/t}-G^{-,u/t}- a_t}{b_t}, \quad \text{defined in } \eqref{def2} .
\]
By Lemma \ref{cut_tie}, along a further subsequence of $\{t_k\}$ (still denoted $\{t_k\}$), we have $Z^{u,v}_{t_k} \todr Y^{u,v}$ for each $u,v \in (0,m)$ as $k \to \infty$ where  $Y^{u,v}$ is an infinitely divisible distribution 
with support including the set $(-\pibarminv_z(u), \pibarpinv_z(v))_*$. 
%%%---begin with contradiction-----
Take $k \to \infty$ on both sides of \eqref{cov5} and apply Fatou's lemma to get
\begin{align}\label{cov5*}
\lim_{k \to \infty}P(^{(r,s)}S_{t_k} > x)&\ge \int_{u,v \in (0, m)} \liminf_{ k \to \infty}  P(Z_{t_k}^{u,v} > x) P(\Gamma_r \in \rmd v, \wt \Gamma_s \in \rmd u)\nonumber \\
                &= \int_{u,v \in (0,m)} P(Y^{u,v} > x) P(\Gamma_r \in \rmd v, \wt \Gamma_s \in \rmd u).
\end{align} 

Let $U_t$ be any L\'evy process with L\'evy measure $\Pi_U$. Define the support of $\Pi_U$ by $S_{\Pi_U}$ and let $c = \inf \{ a > 0 :  S_{\Pi_U} \subset \{ x: |x| \le a\}  \}$. By Sato \cite{sato99} (Theorem 26.1, p.168), for any $\delta > 1/c$ and any $t > 0$, the tail probability of $U_t$ behaves as
\be\label{tailes}
 e^{\delta x \log x }P(|U_t| > x) \to \infty \quad \text{as} \quad x \to \infty. 
\ee
For $u,v \in (0,m)$, $ l = \pibarpinv_z(v) \wedge \pibarminv(u) $. 
Thus $l$ is in the support of the L\'evy measure of $Y^{u,v}$ and $1/l \ge 1/ \pibarinv_z(u)\vee 1/\pibarinv_z(v)$.  We can apply the tail estimate in \eqref{tailes} to $Y^{u,v}$ to get 
\begin{equation}\label{cov10}
\lim_{x \to \infty}  e^{x \log x/l }P(|Y^{u,v}| >x) = \infty.
\end{equation}
It follows from Egorov's theorem that there exists a subset $E$ of the interval $(0,m)$ with positive Lebesgue measure such that \eqref{cov10} holds uniformly on $E$. Multiply $ e^{x\log x /l }$ on both sides of \eqref{cov5*}. Then the modified RHS of \eqref{cov5*} tends to infinity as $x \to \infty$, while the modified LHS of \eqref{cov5*} converges to zero as a result of the estimate
\ben\label{cov11}
   e^{x \log x /l} (2\pi)^{\frac 1 2}\int_x^\infty e^{-y^2/2} \rmd y \le  e^{x\log x/l} O(e^{-x^2/2})
        \to 0 \text{  as } x \to \infty.
\een 
This contradiction proves that $\Pi_z(\cdot) \equiv 0$ and therefore $Z$ is Gaussian.
This means that $Z$ is $N(0,\tau'^2)$ for some $\tau'^2 >0$ (else $Z$ would be degenerate, which case we eliminated earlier). Here we use $'$ to indicate that $\tau'$ depends on the chosen subsequence.
We have shown that for each sequence, there exists a subsequence ${t'}$ such that $S_{t'}\to N(0,\tau'^2)$.
By the assumption in \eqref{convNm+}, we have through this subsequence that $({}^{(r,s)}X_{t'}-a_{t'})/b_{t'}\to N(0,1) $. This forces $\tau'^2 = 1$. 
Since this is true for all subsequences, we have completed the proof for the case when the limit distribution is normal. The proof for ${}^{(r)}\wt X_t$ follows similarly.
\end{proof}

\section{Related Results}\label{sect:aux}
Recall that the quadratic variation process of $X_t$ is defined as $V_t : = \sigma^2 t +  \sum_{s\le t} (\Delta X_s)^2$, and let the trimmed versions of $V_t$ be \[{}^{(r,s)}V_t : = V_t - \sum_{i=1}^r (\Delta X_t^{(i)})^2 -\sum_{j=1}^s (\Delta X_t^{(j),-})^2 \quad {\rm and} \quad
{}^{(r)}\wt V_t : = V_t - \sum_{i=1}^r (\wt{\Delta X}_t^{(i)})^2,
\]
 respectively corresponding to asymmetrical and modulus trimming. We can deduce from Theorem \ref{convNc} the following relationships between the trimmed quadratic variation processes and the untrimmed version.

\begin{corollary}\label{qv}
Under the assumptions of Theorem \ref{convNc}, for any $r,s \in \N$, $b_t > 0$ and $\tau^2 > 0 $,  as $t\dto 0$, 
\begin{equation}\label{qv1}
\qquad \frac{^{(r,s)}V_t}{b^2_t} \topr \tau^2  \quad {\rm or } \quad \frac{^{(r)} \wt V_t}{b^2_t} \topr \tau^2 \quad \text{if and only if} \quad \qquad \frac{V_t}{b^2_t} \topr \tau^2.
\end{equation}
Furthermore, \eqref{qv1} is equivalent to the existence of $a_t \in \R$, $b_t >0$ such that
\be\label{don}
\frac{X_t-a_t}{b_t} \todr N(0, \tau^2), \quad \text{as} \quad t \dto 0.
\ee
The $b_t$ in \eqref{qv1} and \eqref{don} can be chosen to be the same functions.
\end{corollary}

\begin{proof} [Proof of Corollary~\ref{qv}:]
The quadratic variation process of $X_t$ with triplet $(\gamma, \sigma^2, \Pi)$ is a L\'evy subordinator with drift $\sigma^2$ and L\'evy measure $\Pi_q$ where $\pibar_q (x) = \pibar(\sqrt{x})$ for each $x > 0$. Apply Theorem \ref{convNc} to $V_t$ with centering function $0$ and norming function $b^2_t$ to get necessity. Sufficiency is a consequence of Lemma \ref{max1}. This completes the proof of \eqref{qv1}.

The second statement comes from applying the Kallenberg convergence criterion \eqref{cut_KL} for subordinators, which gives that \eqref{don} holds if and only if for each $x >0$, as $t\dto 0$, 
\be\label{don2}
 t\pibar(x b_t) \to 0 \quad \text{and} \quad \frac{tV(xb_t)}{b^2_t} \to \tau^2;
\ee
also that $V_t/b^2_t \topr \tau^2$ holds if and only if for each $x > 0$, as $t \dto 0$,
\be\label{don_s}
  t\pibar_q(x b^2_t) \to 0 \quad \text{and} \quad \frac{t}{b^2_t} \int_{0\le |y| \le xb_t^2 } y \Pi_q(\rmd y) \to \tau^2.
\ee 
Observe that $t\pibar_q(x b_t^2) = t\pibar(\sqrt{x} b_t)$ and 
\[ \frac{t}{b^2_t} \int_{0\le y \le xb_t^2 } y \Pi_q(\rmd y) = \frac{t}{b^2_t} \int_{0\le |y| \le \sqrt{x}b_t } y^2 \Pi(\rmd y) = \frac{tV(\sqrt{x}b_t)}{b^2_t}.
\] Hence the two conditions in \eqref{don2} and \eqref{don_s} are equivalent. This completes the proof. 
\end{proof}

The next corollary gives a subsequential version of Theorem \ref{convNc}. We say that $X_t$ is in the \emph{domain of partial attraction of the normal distribution} if there exist sequences $t_k \dto 0$,  $a_k \in \R$ and $b_k > 0$ such that 
\be\label{DoPN}
\frac{X_{t_k} - a_k}{b_k} \to N(0,1).
\ee A necessary and sufficient condition for \eqref{DoPN} is that 
\[ \liminf_{t \dto 0} \frac{x^2\pibar(x)}{V(x)} = 0.
\]
%A necessary and sufficient condition for \eqref{DoPN} is that 
%\[ \liminf_{t \dto 0} \frac{x^2\pibar(x)}{V(x)} = 0.
%\]
\begin{corollary}\label{sub_CN}
Assume $\pibar(0+) = \infty$. \eqref{DoPN} holds if and only if, for any $r, s \in \N$, there exist sequences  $t_k' \dto 0$, $a_k'$ and $b_k' > 0$  such that 
\be\label{DoPN_1}
\frac{^{(r,s)}X_{t_k'} - a_k'}{b_k'} \to N(0,1), \quad \text{as}\quad k \to \infty,
\ee
or, equivalently, 
\be\label{DoPN_2}
\frac{^{(r)}\wt X_{t_k'} - a_k'}{b_k'} \to N(0,1), \quad \text{as}\quad k \to \infty.
\ee

\end{corollary}

 \begin{proof}
That \eqref{DoPN} implies \eqref{DoPN_1} or \eqref{DoPN_2} is obvious by Lemma \ref{max1}. In this case we can choose the same sequences, i.e., $(t_k') = (t_k)$, $(a_k') = (a_k)$ and $(b_k') = (b_k)$. For the converse, write ${}^{(r,s)}S_{t_k'} = ({}^{(r,s)}X_{t_k'} -a_k')/b_k'$. The convergence of ${}^{(r,s)}S_{t_k'} \todr N(0,1)$ as $k \to \infty$ implies 
the convergence of  $S_{t_k'} \todr N(0,1)$ as $k \to \infty$ can be proved similarly as that of Theorem \ref{convNc} by restricting to a particular subsequence. The same norming and centering sequence can be used. 
\eqref{DoPN_2} implies \eqref{DoPN} can be proved similarly.
\end{proof}

Next, we will give two easy corollaries with degenerate limit distributions.
\begin{corollary}\label{deri0} (Weak Derivative at $0$)
Suppose $\pibar(0+) = \infty$ and $r,s \in \N$. As $t\dto 0$, we have
\begin{equation}\label{deri1}
\frac{X_t}{t} \to \delta \quad \text{if and only if} \quad \frac{{}^{(r,s)}X_t}{t} \to \delta  \quad \text{or} \quad \frac{{}^{(r)}\wt X_t}{t} \to \delta,
\end{equation}
or equivalently as $x \to 0$,
\be\label{deri2}
\sigma^2 =0, \quad x\pibar(x)\to 0,\quad \text{and} \quad \nu(x) \to \delta.
\ee
\end{corollary}
If $X$ is a subordinator, $\delta = \rmd_X$ is the drift coefficient.

\begin{corollary}\label{rs0} (Relative Stability)
Suppose $\pibar(0+) = \infty$ and $r,s \in \N$. As $t\dto 0$, there exists a norming function $b_t \dto 0$ such that
\begin{equation}\label{rs1}
\frac{X_t}{b_t} \to 1 \quad \text{if and only if} \quad \frac{{}^{(r,s)}X_t}{b_t} \to 1  \quad \text{or} \quad \frac{{}^{(r)}\wt X_t}{b_t} \to 1,
\end{equation}
or equivalently as $x \to 0$,
\be\label{rs2}
\sigma^2 =0, \quad \text{and} \quad  \frac{\nu(x)}{x\pibar(x)} \to \infty.
\ee Furthermore, $b_t$ is regularly varying with index $1$.
\end{corollary}

\begin{proof} [Proof of Corollaries~\ref{deri0} and \ref{rs0} :]  These are simple consequences of Theorem \ref{convNc} with degenerate limits.
That the untrimmed version of \eqref{deri1} is equivalent to \eqref{deri2} is proved in Theorem 2.1 of	 Doney and Maller \cite{dm02}.
The equivalence of the untrimmed version of \eqref{rs1} and \eqref{rs2} is proved in Theorem 2.2 of Doney and Maller \cite{dm02}.
\end{proof}

\noindent {\bf Concluding Remarks}
% intermediate trimming
Besides trimming of a bounded number of jumps, there is also theoretical interest in more general trimming where the number of jumps taken away goes to infinity.
The theory of intermediate and heavy trimming is more complex and requires, in general, quite different techniques. The proofs in this paper will not extend immediately to intermediate or heavy trimming cases. In order to tackle those problems, we need arguments along the lines of, for example, Griffin and Pruitt \cite{GriPru1987}, Griffin and Mason \cite{GriMas1991},  Cs\"org\H{o}, Haeusler and Mason (\cite{CHM1988a}, \cite{CHM1991}). Note that in the asymptotic normality case, by Griffin and Pruitt \cite{GriPru1987}, it is essential to restrict to a symmetric marginal distribution. Also see Berkes and Horv\'ath \cite{BH} for more recent developments on trimmed sums.

In the small time paradigm, we zoom in to focus on the hierarchy of the very
small jumps. This promises a fresh perspective in seeking out potential applications. There is an increasing volume of papers from other fields such as physics, chemistry and modern finance, with focal points on instantaneous behaviours of a process. For example, Harris et al. \cite{harris2012} inspect the molecular movement in the blood stream of a certain protein; Zheng et al. \cite{zhengetal2013} study the small time movement of self-propelled Janus particles in a fluid; A\"it-Sahalia and Jacod (e.g.\cite{aitJacod2009}) compute the activity index for highly frequently traded financial data. In particular, and in many other application areas, ``L\'evy flights''€ (processes with heavy tailed increment distributions) are found to accurately describe many physical processes, see for example Davis and Marshak \cite{DavisMarshak1997} on scattering of photons. With increasing power in measurement precision and better data analysis tools, the need for local investigation could become more substantial.

%However, we believe that light trimming as it is stated in this paper is also of separate interest. More general trimming in the L\'evy context is currently pursued by other researchers in the field as well.

\bigskip
\noindent{\bf Acknowledgments}
The author is very grateful to Prof. Ross Maller and Dr. Boris Buchmann for many helpful discussions and critically reading the manuscript.

%---Bibliography---------------

\bibliography{Library_Levy}
\bibliographystyle{abbrv}

\end{document}